\documentclass{amsart}[12pt]
\usepackage{amsmath,amsfonts,amssymb,latexsym,amsthm}
\usepackage{graphicx,epsf}
\usepackage{epsfig}
\usepackage{cyr}
\usepackage{psfrag}
\usepackage{comment}

\newtheorem{thm}{Theorem}
\newtheorem{lemma}[thm]{Lemma}

\newtheorem{corollary}[thm]{Corollary}

\def\.{\hskip.06cm}
\def\ts{\hskip.03cm}

\def\sq{\square}

\def\zz{\mathbb Z}

\def\CC{\mathbb C}

\def\sm{\smallsetminus}
\def\pr{\prime}

\def\la{\lambda}

\def\al{\alpha}
\def\be{\beta}

\def\vp{\varphi}

\def\cf{\mathcal F}
\def\cg{\mathcal G}

\def\ssu{\subset}

\def\<{\langle}
\def\>{\rangle}

\def\syt{ {\text {\rm SYT}  } }

\def\PP{ {\text {\rm \bf P} } }

\def\ts{\hskip.01cm}

\def\0{{\mathbf 0}}

\def\bs{\mathbf{z}}
\def\bt{\mathbf{u}}
\def\cor{\mathcal{C}}
\def\dor{\mathcal{D}}

\begin{document}
\title[Weighted hook length formula]%
    {The weighted hook length formula}
\author[Ionu\c t Ciocan-Fontanine]{Ionu\c t Ciocan-Fontanine$^*$}
\author[Matja\v z Konvalinka]{Matja\v z Konvalinka$^\dagger$}
\author[Igor Pak]{Igor Pak$^\ddag$}
\date{April 8, 2010}

\thanks{${\hspace{-1ex}}^*$School of Mathematics, U. of
Minnesota, Minneapolis; MN 55455, USA~\texttt{ciocan@math.umn.edu}}

\thanks{${\hspace{-.95ex}}^\dagger$Department of Mathematics,
Vanderbilt U., Nashville, TN 37240, USA;~\texttt{matjaz.konvalinka@vanderbilt.edu}}

\thanks{${\hspace{-.95ex}}^\ddag$Department of Mathematics,
UCLA, Los Angeles, CA 90095, USA;~\texttt{(pak@)math.ucla.edu}}

\begin{abstract}
\noindent
Based on the ideas in~\cite{CKP}, we introduce the weighted analogue
of the branching rule for the classical \emph{hook length formula},
and give two proofs of this result.  The first proof is completely
bijective, and in a special case gives a new short combinatorial
proof of the hook length formula.  Our second proof is probabilistic,
generalizing the (usual) hook walk proof of Green-Nijenhuis-Wilf~\cite{GNW},
as well as the $q$-walk of Kerov~\cite{Ker1}.  Further applications are also presented.
\end{abstract}

\maketitle

\section*{Introduction}\label{s:intro}

\noindent
The classical hook length formula gives a short product formula
for the dimensions of irreducible representations of the
symmetric group, and is a fundamental result in
algebraic combinatorics.  The formula was discovered by Frame,
Robinson and Thrall in~\cite{FRT} based on earlier results
of Young \cite{You}, Frobenius \cite{Fro} and Thrall \cite{Thr}.  Since then, it has been reproved,
generalized and extended in several different ways, and applied in a number of fields ranging from algebraic
geometry to probability, and from group theory to the analysis of algorithms.
Still, the hook length formula remains deeply mysterious and its
full depth is yet to be completely understood.  This paper is a new
contribution to the subject, giving a new multivariable extension of the formula, and a new combinatorial proof associated with it.

Let $\la = (\la_1 \ge \la_2 \ge \ldots)$ be a partition
of~$n$, let $[\la]$ be the corresponding Young diagram,
and let $\syt(\la)$ denote the set of standard
Young tableaux of shape~$\la$ (full definitions will be
given in the next section).
The \emph{hook length formula} for the dimension of the
irreducible representation $\pi_\la$ of the symmetric
group~$S_n$ can be written as follows:
$$\text{(HLF)} \ \qquad
\dim \pi_\la \, = \, |\syt(\la)| \, = \,
\frac{n!}{\prod_{x\in [\la]} \. h_x} \,,
$$
where the first equality is A.~Young's combinatorial
interpretation, the product on the right is
over all squares~$x$ in the Young diagram corresponding
to partition~$\la$, and~$h_x$ are the hook numbers
(see below).  In fact, Young's original approach
to the first equality hints at the direction of the
proof of the second equality.  More precisely, he proved the
following \emph{branching rules}:
$$\text{(BR)} \qquad \dim \pi_\la \. = \. \sum_{\mu \to \la} \. \dim \pi_\mu
\quad \ \text{and} \quad
|\syt(\la)| \. = \. \sum_{\mu \to \la} \. |\syt(\mu)|\,,
$$
where the summation is over all partitions~$\mu$
of~$n-1$ whose Young diagram fits inside that of~$\la$
(the second branching rule is trivial, of course).
Induction now implies the first equality in~(HLF).

In a similar way, the hook length formula is equivalent
to the following \emph{branching rule for the hook lengths}:
$$\text{(BRHL)} \qquad
\sum_{\text{corner} \, \. (r,s) \. \in \. [\la]}  \ \, \frac{1}{n}
\,\. \prod_{i=1}^{r-1} \, \frac{h_{is}}{h_{is}-1}
\, \prod_{j=1}^{s-1} \, \frac{h_{rj}}{h_{rj}-1}
\,  = \. 1\..
$$
Although this formula is very natural, it is difficult to
prove directly, so only a handful of proofs employ it
(see below and Subsection~\ref{ss:rem-hist}).

In an important development, Green, Nijenhuis and Wilf introduced
the \emph{hook walk} which proves~(BRHL) by a combination of a
probabilistic and a short but delicate induction argument~\cite{GNW}.
Zeilberger converted this hook walk proof into a bijective proof of~(HLF)~\cite{Zei},
but lamented on the ``enormous size of the input and output'' and ``the
recursive nature of the algorithm'' (ibid,~$\S 3$).  With time,
several variations of the hook walk have been discovered, most notably
the $q$-version of Kerov~\cite{Ker1}, and its further generalizations and variations (see~\cite{CLPS,GH,Ker2}). Still, before
this paper,  there were no direct combinatorial proofs of~(BRHL).

\medskip

In this paper we introduce and study the following \emph{weighted
branching rule for the hook lengths}:
$$\text{(WHL)} \qquad \
\aligned
& \sum_{\text{corner} \, (r,s) \. \in \. [\la]}  \ x_r \ts y_s \, \
\prod_{i=1}^{r-1} \. \left(1+ \frac{x_i}{x_{i+1}+\ldots + x_{r} \ts + \ts
 y_{s+1} + \ldots + y_{\lambda_i}}\right) \\
& \hskip1.cm \times \, \prod_{j=1}^{s-1} \.  \left(1+ \frac{y_j}{x_{r+1}+\ldots + x_{\lambda'_j} \ts + \ts
  y_{j+1} + \ldots + y_{s}}\right) \.  = \ts \sum_{(i,j) \in [\la]} \. x_i \ts y_j \..
\endaligned
$$
Here the weights $x_1,x_2,\ldots$ and
$y_1,y_2, \ldots$ correspond to the rows and columns of the Young diagram,
respectively, so the weight of square $(i,j)$ is $x_i y_j$.   Note that (WHL)
becomes (BRHL) for the unit weights $x_i = y_j = 1$, and can
be viewed both as a probabilistic result (when the weights are positive),
and as a rational function identity (when the weights are formal
commutative variables).

There is an interesting story behind this formula, as a number of its
special cases seem to be well known.  Most notably, for the staircase shaped
diagrams, Vershik discovered the formula and proved it by a technical
inductive argument~\cite{Ver}.  In this case, an elegant Lagrange interpolation
argument was later found by Kirillov~\cite{Kir} (see also~\cite{Ban,Ker2}),
while an algebraic application and a hook walk style proof was recently given
by the authors in~\cite{CKP}. In a different direction, there is a standard
(still multiplicative) $q$-analogue of~(HLF), which can be obtained as the
branching rule for the Hall-Littlewood polynomials (see~\cite[$\S 3$]{Mac}
for the explicit formulas and references).
\medskip

\noindent
There are three main tasks in the paper:

\medskip

\quad $(1)$ \ give a direct bijective proof of~(BRHL),

\medskip

\quad $(2)$ \ prove a weighted analogue~(WHL), and

\medskip

\quad $(3)$ \ give a hook walk proof of~(WHL).

\medskip

\noindent
Part~$(1)$ is done in Section~\ref{s:bij} and is completely self-contained.
Part~$(2)$ is essentially a simple extension of part~$(1)$, based on certain
properties of the bijection. The bijection in~$(1)$ is robust enough to prove
several variations on~(BRHL), which all have weighted analogues (Section~\ref{s:wbr}).
In a special case this gives certain Kirillov's summation formulas
and Kerov's $q$-formulas in~\cite{Ker1}, which
until now had only analytic proofs.

In Section~\ref{s:walk} we define two new walks, a ``weighted''
and a ``modified'' hook walk.  While both can be viewed as
extensions of the usual hook walk, we show that the latter
reduces to the former.  In fact, the modified hook walk is
motivated and implicitly studied in our previous paper~\cite{CKP}.
The complete proof of~(WHL) via the weighted hook walk is then
given in Section~\ref{s:walk-proof}.  We conclude with historical
remarks and final observations in Section~\ref{s:rem}.

\newpage

\section{Definitions and notations}\label{s:def}

An integer sequence \. $\la = (\la_1,\la_2,\dots,\la_\ell)$ \.
is a {\it partition} \. of~$n$, write $\la \vdash n$,
if $\la_1 \ge \la_2 \ge \ldots \ge \la_\ell >0 $, and
$|\la| = \la_1+\la_2 + \ldots +\la_\ell = n$.  From now on, let $\ell=\ell(\la)$
denote the number of parts, and let $m = \la_1$
denote the length of the largest part of~$\la$.
Define the \ts {\it conjugate partition} \.
$\la^\pr = (\la_1^\pr,\dots,\la_m^\pr)$ \. by
$\la_j^\pr = |\{i : \la_i \ge j\}|$.

A {\it Young diagram} \, $[\la]$ \. corresponding to $\la$
is a collection of squares $(i,j) \in \zz^2$, such that $1 \le j \le \la_i$.
The \emph{hook} $H_{\bs} \ssu [\la]$ is the set of squares weakly
to the right and below of $\bs = (i,j) \in [\la]$, and the {\it hook length} \.
$h_\bs = |H_\bs|= \la_i +\la_j^\pr - i - j +1$ is the size of the hook
(see Figure~\ref{f:hook}).

We say that \, $(i_1,j_1) \prec (i_2,j_2)$ \, if \, $i_1 \le i_2$,
$j_1 \le j_2$, and $(i_1,j_1) \ne (i_2,j_2)$.
A {\it standard Young tableau}~$A$ \, of shape $\la$ is a bijective map
$f: [\la] \to [n] = \{1,\dots,n\}$, such that $f(i_1,j_1) < f(i_2,j_2)$
for all $(i_1,j_1) \prec (i_2,j_2)$. We denote the set of
standard Young tableaux of shape~$\la$  by~$\syt(\la)$.  For example,
for $\la=(3,2,2)\vdash 7$, the \emph{hook length formula} (HLF) in the
introduction gives:
$$|\syt(3,2,2)| \, = \,
\frac{7!}{5\cdot 4 \cdot 3 \cdot 2\cdot 2 \cdot 1 \cdot 1} \, = \, 21\ts.
$$

Throughout the paper, we draw a Young diagram with the first
coordinate increasing downwards, and the second coordinate increasing
from left to right.  We then label the rows of the diagram with variables~$x_1,x_2,\ldots$,
and the columns with variables~$y_1,y_2,\ldots$ (see Figure~\ref{f:hook}).
Thus, if the reader prefers the French notation
(and standard Descartes coordinates), then a $90^\circ$ counterclockwise
rotation of a diagram is preferable to the mirror reflection as suggested
in~\cite{Mac}.

\begin{figure}[hbt]
\psfrag{x1}{$x_1$}
\psfrag{x2}{$x_2$}
\psfrag{x3}{$x_3$}
\psfrag{x4}{$x_4$}
\psfrag{x5}{$x_5$}
\psfrag{y1}{$y_1$}
\psfrag{y2}{$y_2$}
\psfrag{y3}{$y_3$}
\psfrag{y4}{$y_4$}
\psfrag{y5}{$y_5$}
\psfrag{y6}{$y_6$}
\psfrag{L}{$\la$}
\psfrag{A}{$A$}
\begin{center}
\epsfig{file=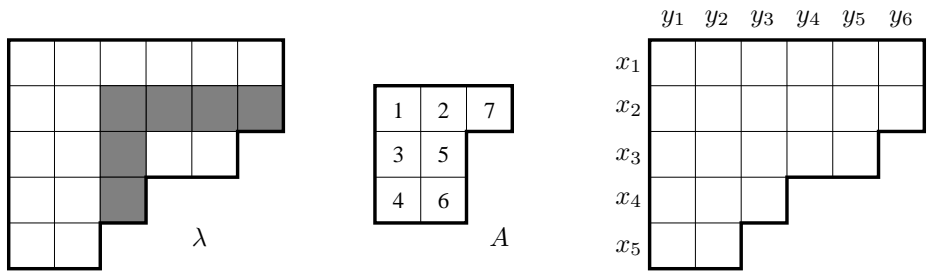,height=3.5cm}
\end{center}
\caption{Young diagram $[\la]$,  $\la = (6,6,5,3,2)$, and a
hook $H_{23}$ with hook length $h_{23}=6$; a standard Young tableau~$A$ of
shape $(3,2,2)$; a labeling of rows and columns of~$\la$.}
\label{f:hook}
\end{figure}

A \emph{corner} of the Young diagram~$[\la]$ is a square $(i,j) \in [\la]$
such that $(i+1,j) \notin [\la]$, $(i,j+1) \notin [\la]$.  Clearly, $(i,j)\in [\la]$ is a corner
if and only if $h_{ij} =1$. By $\cor[\la]$ we denote the set
of corners of~$[\la]$.  For example, the diagram $[3,2,2]$ has two corners, $(1,3)$ and $(3,2)$.

As in the introduction, we write $\mu \to \la$ for all $|\mu| = |\la|-1$
such that $[\mu] \ssu [\la]$.  Alternatively, this is equivalent to saying
that $[\mu]=[\la] \sm \bs$, for some corner $\bs \in \cor[\la]$.   Now the
\emph{branching rule}~(BR) for the standard Young tableaux follows
immediately by removing the corner containing~$n$.

\vskip.8cm


\section{A new bijective proof of the hook length formula}\label{s:bij}

\subsection{The algebraic setup} \,
We start by formalizing the induction approach
outlined in the introduction.  First, observe that to obtain the
hook length formula~(HLF) by induction it suffices to prove
the following identity:
\begin{equation} \label{e:ind-hl}
\frac{n!}{\prod_{\bs\in [\la]} h_\bs} \, = \, \sum_{\mu \to \la} \.
\frac{(n-1)!}{\prod_{\bt\in [\mu]} h_\bt}\,.
\end{equation}
Indeed, by the branching rule~(BR) for the standard Young tableaux,
this immediately gives the induction step:
$$
|\syt(\la)| \, = \, \sum_{\mu \to \la} \. |\syt(\mu)| \, = \,
\sum_{\mu \to \la} \.
\frac{(n-1)!}{\prod_{\bt\in [\mu]} h_\bt} \, = \,
\frac{n!}{\prod_{\bs\in [\la]} h_\bs}\,,
$$
which proves the~(HLF).  Rewriting~\eqref{e:ind-hl}, we obtain:
\begin{equation} \label{e:ind-hl-1}
1 \, = \, \sum_{\mu \to \la}
\frac{(n-1)!}{n!} \frac{\prod_{\bs \in \la} \. h_\bs}{\prod_{\bt \in [\mu]} \. h_\bt} \. = \.
\sum_{(r,s)\in \cor[\la]} \. \frac{1}{n} \. \prod_{i=1}^{r-1} \. \frac{h_{is}}{h_{is}-1}
\. \prod_{j=1}^{s-1} \. \frac{h_{rj}}{h_{rj}-1}\,.
\end{equation}
Multiplying both sides of~\eqref{e:ind-hl-1} by the common denominator, we get the following
equivalent identity:
\begin{equation} \label{e:brhl}
n \cdot \prod_{\bs \in [\la]\sm \cor[\la]} \, (h_\bs-1) \, = \,
\sum_{(r,s)\in \cor[\la]} \.  \prod_{i=1}^{r-1} \. h_{is}
\. \prod_{j=1}^{s-1} \. h_{rj} \, \prod_{\bs \in \dor_{rs}[\la]}
(h_\bs-1) \,,
\end{equation}
where the last product is over the set
$$\dor_{rs}[\la] \, = \, \{(i,j)\in [\la]\sm \cor[\la], \
\text{such that} \ i \ne r, \, j \ne s\}\ts .
$$
Below we prove the following multivariable extension of this identity:
\begin{equation} \label{e:wbrhl}
\aligned
& \left[\sum_{(p,q) \in [\la]} \.x_p \ts y_q\right] \. \cdot \. \left[
\prod_{(i,j) \in [\la]\sm \cor[\la]} \, \left(x_{i+1} + \ldots + x_{\la_j'}
+ \. y_{j+1}+ \ldots + y_{\la_i}\right)\right] \\
& \, = \, \sum_{(r,s) \in \cor[\la]} \.x_r \ts y_s \,  \left[
\prod_{(i,j) \in \dor_{rs}[\la]} \, \left(x_{i+1} + \ldots + x_{\la_j'}
+ \. y_{j+1}+ \ldots + y_{\la_i}\right)\right]\\
& \, \times \,  \left[ \prod_{i=1}^{r-1} \, \left(x_i + \ldots + x_{r}
+ \. y_{s+1}+ \ldots + y_{\la_i}\right)\right]  \cdot
\left[ \prod_{j=1}^{s-1} \, \left(y_j + \ldots + y_{s}
+ \. x_{r+1}+ \ldots + x_{\la_j'}\right)\right]
\endaligned
\end{equation}
Clearly, when
$x_1=x_2 = \ldots = y_1 = y_2 = \ldots =1$, we obtain~\eqref{e:brhl}.\footnote{In fact,
equation~\eqref{e:wbrhl} immediately implies~(WHL), but more on this in the next section.}
Note also that both sides are homogenous polynomials
of degree $d_\la=|\la|+2- \bigl|\cor[\la]\bigr|$.

\subsection{The bijection} \,	
Now we present a bijective proof of~\eqref{e:wbrhl}, by interpreting
both sides as certain sets of arrangements of labels (see Section~\ref{s:def}).

For the l.h.s. of~\eqref{e:wbrhl}, we are given:
\begin{itemize}
 \item special labels $x_p,y_q$, corresponding to the first summation
$\sum_{(p,q) \in [\la]} \.x_p \ts y_q$;
 \item a label $x_k$ for some $i < k \le \lambda'_j$, or~$y_l$ for some $j < l \le \lambda_i$,
 in every non-corner square $(i,j)$.
\end{itemize}
Denote by~$F$ the resulting
arrangement of~$d_\la$ labels (see Figure~\ref{f:wex}, first diagram),
and by $\cf_\la$ the set of such labeling arrangements~$F$.

For the r.h.s. of~\eqref{e:wbrhl}, we are given
\begin{itemize}
 \item special labels $x_r,y_s$, corresponding to the corner $(r,s)$;
 \item a label $x_k$ for some $i < k \le \lambda'_j$, or~$y_l$ for some $j < l \le \lambda_i$,
 in every non-corner square $(i,j)$, $i \ne r$, $j \ne s$;
 \item a label $x_k$ for some $i \le k \le \lambda'_j$, or~$y_l$ for some $s < l \le \lambda_i$,
 in every non-corner square $(i,s)$;
 \item a label $x_k$ for some $r < k \le \lambda'_j$, or~$y_l$ for some $j \le l \le \lambda_i$,
 in every non-corner square $(r,j)$.
\end{itemize}
Denote by~$G$ the resulting
arrangement of~$d_\la$ labels (see Figure~\ref{f:wex}, last diagram),
and by $\cg_\la$ the set of such labeling arrangements~$G$. The bijection $\vp: F\mapsto G$ is now defined by rearranging
the labels.

\medskip

\noindent
\emph{Direct bijection~$\vp:\cf_\la \to \cg_\la$.} \,

\noindent
We can interpret the special labels $x_p,y_q$ as the starting square $(p,q)$.
Furthermore, we can interpret all other labels as arrows pointing to a square
in the hook. More specifically, if the label in square $(i,j)$
is $x_{k}$, the arrow points to $(k,j)$, and if the label is $y_{l}$, the arrow
points to~$(i,l)$.

Let the arrow from square $(p,q)$ point to a square $(p',q')$ in the hook $H_{pq} \setminus \{(p,q)\}$, the
arrow from~$(p',q')$ point to a square $(p'',q'')\in H_{p'q'} \setminus \{(p',q')\}$, etc.
Iterating this, we eventually obtain a \emph{hook walk}~$W$ which reaches a corner
$(r,s) \in \cor[\la]$ (see Figure~\ref{f:wex}, second diagram).

Shade row $r$ and column $s$. Now we shift the labels in the hook walk and in its projection onto the shaded row and column.
If the hook walk has a horizontal step from $(i,j)$ to $(i,j')$, move the label in $(i,j)$
right and down from $(i,j)$ to $(r,j')$, and the label from $(r,j)$ up to $(i,j)$. If the hook walk has a
vertical step from $(i,j)$ to $(i',j)$, move the label from $(i,j)$ down and right to $(i',s)$,
and the label from $(i,s)$ left to $(i,j)$. Finally, move the label $x_p$ to $(p,s)$, the
label $y_q$ to $(r,q)$, the label $x_r$ to $(r,0)$, and the label $y_s$ to $(0,s)$. See Figure~\ref{f:wex}, third diagram. We denote by $G$ the resulting arrangement of labels (Figure~\ref{f:wex}, fourth diagram).

We now have labels in all non-corner squares, and special labels $x_r$ and $y_s$ corresponding to the corner $(r,s)$.
We claim that $G \in \cg_\la$. Indeed, if there is a horizontal step in the hook walk
from $(i,j)$ to $(i,j')$, that means that the label in $(i,j)$ is $y_{j'}$, and then
the new label in $(r,j')$ is $y_{j'}$; since the label in that square should be $x_k$ for some $r < k \le \lambda'_j$, or~$y_l$ for some $j' \le l \le \lambda_i$, this is acceptable. Also, the new label in $(i,j)$ is the old label from $(r,j)$, so it is either
$x_k$ for $k > r \geq i$ or $y_l$ for $l > j$; both are acceptable. The case when the step is vertical is analogous.

\begin{figure}[hbt]
\psfrag{x1}{$\scriptstyle x_1$}
\psfrag{x2}{$\scriptstyle x_2$}
\psfrag{x3}{$\scriptstyle x_3$}
\psfrag{x4}{$\scriptstyle x_4$}
\psfrag{x5}{$\scriptstyle x_5$}
\psfrag{x6}{$\scriptstyle x_6$}
\psfrag{y1}{$\scriptstyle y_1$}
\psfrag{y2}{$\scriptstyle y_2$}
\psfrag{y3}{$\scriptstyle y_3$}
\psfrag{y4}{$\scriptstyle y_4$}
\psfrag{y5}{$\scriptstyle y_5$}
\psfrag{y6}{$\scriptstyle y_6$}
\psfrag{y7}{$\scriptstyle y_7$}

\begin{center}
\includegraphics[width=0.95 \textwidth]{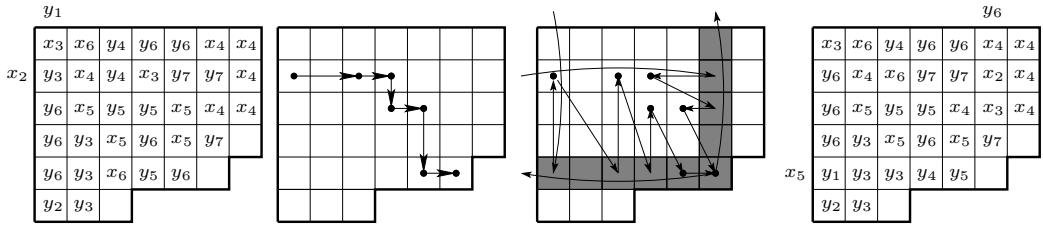}
\end{center}
\caption{An example of an arrangement corresponding to the left-hand side of WBR for $\lambda = 777763$; hook walk; shift of labels; final arrangement.}
\label{f:wex}
\end{figure}

\begin{lemma}
 The map $\vp:  \cf_\la \to \cg_\la$ defined above is a bijection.
\end{lemma}

The lemma follows from the construction of the inverse map.

\medskip

\noindent
\emph{Inverse bijection~$\vp^{-1}:\cg_\la \to \cf_\la$.} \,

\noindent
Start with~$G$ and shade the row and column of~$[\la]$ corresponding
to the two special labels $x_r$ and~$y_s$, where $(r,s)$ is the given corner.
Recall from the construction of~$\vp$
that the projections of~$W$ onto the shaded row are the squares $(r,j)$
with label $y_j$, and the projections of~$W$ onto the shaded column are the squares $(i,s)$
with label $x_i$. Clearly, the smallest such $i$ and $j$ give the special
labels $x_p$, $y_q$ (if no such $i$ and/or $j$ exists, take $p=r$ and/or $q=s$).
Suppose that the label in square $(p,q)$ is $x_{k}$ for $k > p$.
If $k \leq r$, then $x_{k}$ is an acceptable label for the square
$(p,s)$ (and not for $(r,q)$). If $k > r$, then it is an acceptable label for
$(r,q)$ (and not for $(p,s)$). On the other hand, if the label
in $(p,q)$ is $y_{l}$ for $l>q$, then $y_{l}$ is an acceptable label
for $(r,q)$ if $l \leq s$ and an acceptable label for $(p,s)$ if $l>s$.
Therefore, the label at $(p,q)$ determines in which
direction from~$(p,q)$ the step of the walk~$W$ is made.

Now find the next square in that direction whose projections onto shaded row and column
are in the projections of $W$, and repeat the procedure.  At the end we obtain the whole walk~$W$.
Then simply undo the shifting of labels described in the
construction of~$\vp$.

A straightforward check shows that this is indeed the
initial label arrangement~$F$.  This implies the lemma
and completes the proof of \eqref{e:wbrhl} and of the hook length formula~(HLF).
\hfill $\sq$


\vskip1.4cm

\section{Weighted branching rule for the hook lengths}\label{s:wbr}

\subsection{Main theorem.} The main result of this paper can be summarized in one theorem:

\begin{thm} \label{t:main}
 Fix a partition $\lambda$. For commutative variables $x_i,y_j$, write
 $${\prod}_{rs} =  \ x_r \, y_s \prod_{i = 1}^{r-1} {\textstyle \left(1+ \frac{x_i}{x_{i+1}+\ldots + x_{r} + y_{s+1} + \ldots + y_{\lambda_i}}\right)} \cdot \prod_{j = 1}^{s-1} {\textstyle \left(1+ \frac{y_j}{x_{r+1}+\ldots + x_{\lambda'_j} + y_{j+1} + \ldots + y_{s}}\right)}.$$
 Then we have the following rational function identities:
 \begin{enumerate}
  \renewcommand{\labelenumi}{(\alph{enumi})}
  \item $\sum_{(r,s)} {\prod}_{rs} = \sum_{(p,q) \in [\lambda]}  x_p  y_q$
  \item $\sum_{(r,s)} \frac{1}{x_{r+1}+\ldots+x_{\ell(\lambda)}+y_1+ \ldots + y_s} \cdot {\prod}_{rs} = \sum_{p=1}^{\ell(\lambda)} x_p$
  \item $\sum_{(r,s)} \frac{1}{x_{1}+\ldots+x_{r}+y_{s+1}+ \ldots + y_{\lambda_1}} \cdot {\prod}_{rs} = \sum_{q=1}^{\lambda_1} y_q$
  \item $\sum_{(r,s)} \frac{1}{(x_{r+1}+\ldots+x_{\ell(\lambda)}+y_1+ \ldots + y_s)(x_{1}+\ldots+x_{r}+y_{s+1}+ \ldots + y_{\lambda_1})} \cdot {\prod}_{rs} = 1$
 \end{enumerate}
\end{thm}
\begin{proof}
 It is clear that we get part (a) from equation \eqref{e:wbrhl} by dividing both sides by the expression $\prod_{(i,j) \in [\la]\sm \cor[\la]} \, (x_{i+1} + \ldots + x_{\la_j'}
 + \. y_{j+1}+ \ldots + y_{\la_i})$. Identity (b) is equivalent to
 \begin{equation}  \label{e:wbrh2}
 \aligned
 & \quad \left[\sum_{p = 1}^{\ell(\lambda)} \.x_p \right] \. \cdot \. \left[
 \prod_{(i,j) \in [\la]\sm \cor[\la]} \, \left(x_{i+1} + \ldots + x_{\la_j'}
 + \. y_{j+1}+ \ldots + y_{\la_i}\right)\right] \\
 & \quad \, = \, \sum_{(r,s) \in \cor[\la]} \.x_r \,  \left[
 \prod_{(i,j) \in \dor_{rs}[\la]} \, \left(x_{i+1} + \ldots + x_{\la_j'}
 + \. y_{j+1}+ \ldots + y_{\la_i}\right)\right]\\
 & \, \times \,  \left[ \prod_{i=1}^{r-1} \, \left(x_i + \ldots + x_{r}
 + \. y_{s+1}+ \ldots + y_{\la_i}\right)\right]  \cdot
 \left[ \prod_{j=2}^{s} \, \left(y_j + \ldots + y_{s}
 + \. x_{r+1}+ \ldots + x_{\la_j'}\right)\right]
 \endaligned
 \end{equation}

 Let us show that by analogy with \eqref{e:wbrhl}, this identity can be proved by using the bijection $\varphi$.
 The left-hand side of~\eqref{e:wbrh2} corresponds to arrangements as in the left-hand side of \eqref{e:wbrhl}
 with an additional label $x_p$. Similarly, the right hand
 side of~\eqref{e:wbrh2} corresponds to arrangements as in the right-hand side of \eqref{e:wbrhl}, except the square $(r,1)$ does not get a label. Start
 the hook walk in square $(p,1)$ and proceed as in the proof of \eqref{e:wbrhl}. Now observe that the bijection $\varphi$ gives
 the bijection between these sets of label arrangements. We omit the easy details.

 Identity (c) follows from (b) by conjugation, and (d) can be rewritten in the following form:
  \begin{equation}  \label{e:wbrh3}
  \aligned
  & \quad \left[
  \prod_{(i,j) \in [\la]\sm \cor[\la]} \, \left(x_{i+1} + \ldots + x_{\la_j'}
  + \. y_{j+1}+ \ldots + y_{\la_i}\right)\right] \\
  & \quad \, = \, \sum_{(r,s) \in \cor[\la]} \  \left[
  \prod_{(i,j) \in \dor_{rs}[\la]} \, \left(x_{i+1} + \ldots + x_{\la_j'}
  + \. y_{j+1}+ \ldots + y_{\la_i}\right)\right]\\
  & \, \times \,  \left[ \prod_{i=2}^{r} \, \left(x_i + \ldots + x_{r}
  + \. y_{s+1}+ \ldots + y_{\la_i}\right)\right]  \cdot
  \left[ \prod_{j=2}^{s} \, \left(y_j + \ldots + y_{s}
  + \. x_{r+1}+ \ldots + x_{\la_j'}\right)\right]
  \endaligned
  \end{equation}
 We prove \eqref{e:wbrh3} in a similar way. Start the walk in square $(1,1)$ and proceed as above. Observe that in this case, we
 do not get a label in squares $(r,1)$ and $(1,s)$. The bijection $\varphi$, restricted to this set of label arrangements, proves the equality. We omit the easy details.
\end{proof}

\subsection{The $q$-version.} In \cite{Ker1}, Kerov proved the following identities.\footnote{Let us note that this is a corrected version of the theorem as the original contained a typo.}

\begin{corollary}[Kerov]
 Fix a pair of sequences of reals $X_1,\ldots,X_d$ and $Y_0,\ldots,Y_d$ such that $Y_0 < X_1 < Y_1 < X_2 < \ldots < X_d < Y_d$. Define
 \begin{alignat*}{2}
 \pi_k(q) = & \prod_{i=1}^{k-1} \frac{q^{Y_i}-q^{X_k}}{q^{X_i}-q^{X_k}} \prod_{i=k+1}^d \frac{q^{X_k}-q^{Y_{i-1}}}{q^{X_k}-q^{X_i}}, \qquad 1 \leq k \leq d \\
 Z = &\sum_{i=1}^d q^{X_i} - \sum_{i=1}^{d-1}q^{Y_i}, \qquad S = \sum_{1 \leq i \leq j \leq d} (q^{Y_{i-1}}-q^{X_i})(q^{X_j}-q^{Y_j}).
 \end{alignat*}
 Then:
 \begin{enumerate}
  \renewcommand{\labelenumi}{(\alph{enumi})}
  \item $\displaystyle\sum_{k} \pi_k(q) = 1$
  \item $\displaystyle \sum_{k} \frac{q^{Y_0}-q^{X_k}}{q^{Y_0} - Z} \, \pi_k(q) = 1$
  \item $\displaystyle \sum_{k} \frac{q^{X_k}-q^{Y_d}}{Z - q^{Y_d}} \, \pi_k(q) = 1$
  \item $\displaystyle \sum_{k} \frac{(q^{Y_0}-q^{X_k})(q^{X_k}-q^{Y_d})}{S} \, \pi_k(q) = 1$
 \end{enumerate}
\end{corollary}
\begin{proof}
 The formulas follow by setting
 $$x_i = q^{X_i} - q^{Y_{i-1}} \qquad y_j = q^{Y_{d+1-j}}-q^{X_{s+1-j}}$$
 and taking equations (a)--(d) from Theorem \ref{t:main} for the staircase partition $\lambda = (d,d-1,\ldots,1)$.

 In (a), let $r = k$, $s = d+1-k$, $\lambda_i = d+1-i$, $\lambda'_j = d+1-j$. We have
 $$x_{i+1}+\ldots + x_{r} = \left(q^{X_{i+1}} - q^{Y_{i}}\right)+ \ldots + \left(q^{X_k}  - q^{Y_{k-1}}\right)$$
 and
 $$y_{s+1} + \ldots + y_{\lambda_i} = \left(q^{Y_{i}} - q^{X_{i}}\right) + \ldots + \left(q^{Y_{k-1}}  - q^{X_{k-1}}\right).$$
 That means that
 $$\prod_{i = 1}^{r-1} {\textstyle \left(1+ \frac{x_i}{x_{i+1}+\ldots + x_{r} + y_{s+1} + \ldots + y_{\lambda_i}}\right)} = \prod_{i = 1}^{k-1} {\textstyle \left(1+ \frac{q^{X_i} - q^{Y_{i-1}}}{q^{X_k}-q^{X_{i}}}\right)} = \prod_{i = 1}^{k-1} {\textstyle \frac{q^{Y_{i-1}} - q^{X_{k}}}{q^{X_i}-q^{X_{k}}}}.$$
 Similarly,
 $$\prod_{j = 1}^{s-1} {\textstyle \left(1+ \frac{y_j}{x_{r+1}+\ldots + x_{\lambda'_j} + y_{j+1} + \ldots + y_{s}}\right)} = \prod_{j=1}^{d-k} {\textstyle \frac{q^{Y_{d+1-j}} - q^{X_k}}{q^{X_{d+1-j}}-q^{X_{k}}}} =  \prod_{i=k+1}^d {\textstyle \frac{q^{X_k} - q^{Y_{i}}}{q^{X_{k}}-q^{X_i}}}.$$
 We also have
 $$\textstyle \frac{x_r y_s}{(x_{r+1}+\ldots+x_{\ell(\lambda)}+y_1+ \ldots + y_s)(x_{1}+\ldots+x_{r}+y_{s+1}+ \ldots + y_{\lambda_1})} = \frac{\left(q^{X_{k}}-q^{Y_{k-1}}\right)\left(q^{Y_{k}}-q^{X_{k}}\right)}{\left(q^{Y_{d}}-q^{X_{k}}\right)\left(q^{X_{k}}-q^{Y_{0}}\right)}.$$
 Together with the identity (d) in Theorem \ref{t:main}, this implies
 $$1 = \sum_{k=1}^d {\textstyle \frac{\left(q^{Y_{k-1}}-q^{X_{k}}\right)\left(q^{X_{k}}-q^{Y_{k}}\right)}{\left(q^{X_{k}}-q^{Y_{d}}\right)\left(q^{Y_{0}}-q^{X_{k}}\right)}} \prod_{i = 1}^{k-1} {\textstyle \frac{q^{Y_{i-1}} - q^{X_{k}}}{q^{X_i}-q^{X_{k}}}} \prod_{i=k+1}^d {\textstyle \frac{q^{X_k} - q^{Y_{i}}}{q^{X_{k}}-q^{X_i}}} = $$
 $$= \sum_{k=1}^d \prod_{i = 1}^{k-1} {\textstyle \frac{q^{Y_{i}} - q^{X_{k}}}{q^{X_i}-q^{X_{k}}}} \prod_{i=k+1}^d {\textstyle \frac{q^{X_k} - q^{Y_{i-1}}}{q^{X_{k}}-q^{X_i}}} = \sum_{k=1}^d \pi_k(q),$$
 as desired. The proof of identities (b)--(d) follows the same lines.
\end{proof}

\section{Weighted and modified hook walks}\label{s:walk}

\subsection{Weighted hook walk.} Fix a partition $\lambda$ and positive weights $x_1,\ldots,x_{\lambda'_1},y_1,\ldots,y_{\lambda_1}$. Consider
the following combinatorial random process. Select the starting square $(i,j) \in [\lambda]$ with probability
proportional to $x_i y_j$. At each step, move from square $(i,j)$ to a
random square in $H_{ij} \setminus \{(i,j)\}$ so that the probability of moving to the square $(k,j)$, $i < k \leq \lambda'_j$, is proportional
to $x_{k}$, and the probability of moving to the square $(i,l)$, $j < l \leq \lambda_i$, is proportional to $y_{l}$. When we reach a corner,
the process ends. We call this a \emph{weighted hook walk}.

\begin{thm} \label{thm:whw}
 The probability that the weighted hook walk stops in the corner $(r,s)$ of $\lambda$ is equal to
 $$\frac {x_r \ts y_s}{\sum_{(p,q) \in [\la]} \. x_p \ts y_q}  
 \prod_{i=1}^{r-1} \. \left({\textstyle 1+ \frac{x_i}{x_{i+1}+\ldots + x_{r} \ts + \ts
  y_{s+1} + \ldots + y_{\lambda_i}}}\right) 
  \prod_{j=1}^{s-1} \. \left({\textstyle 1+ \frac{y_j}{x_{r+1}+\ldots + x_{\lambda'_j} \ts + \ts
  y_{j+1} + \ldots + y_{s}}}\right)$$
\end{thm}

Note that the sum of these products over all $(r,s) \in \cor[\lambda]$ is equal to the ratio of the left-hand side and the right-hand side of Theorem~\ref{t:main}, part~(a). Since the sum of these probabilities over all corners is equal to $1$, we conclude that Theorem~\ref{thm:whw} implies~(WHL). We prove Theorem \ref{thm:whw} in the next section by an inductive argument. From above, this gives an alternative proof of~(WHL).

\subsection{Modified weighted hook walk.} Take a square $(i,j)$ in $[\la]$, and
find the corner $(r_1,s_1)$ with the smallest $r_1$ satisfying $r_1 \geq i$, and the
corner $(r_2,s_2)$ with the smallest $s_2$ satisfying $s_2 \geq j$. The \emph{modified hook}
is the set $\{ (k,j) \colon r_1 < k \leq \lambda'_j\} \cup \{ (i,l) \colon s_2 < l \leq \lambda_i \}$.
An example is given in Figure \ref{f:mod}.

\begin{figure}[hbt]
\begin{center}
\epsfig{file=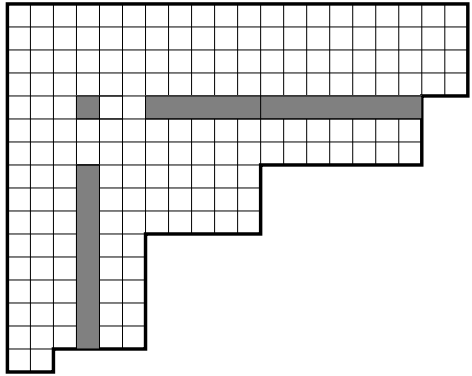,height=5cm}
\end{center}
\caption{The square $(5,4)$ of the diagram and its modified hook of length~$20$ in the partition $(20,20,20,20,18,18,18,11,11,11,6,6,6,6,6,2)$.}
\label{f:mod}
\end{figure}

Recall that we have positive weights $x_1,\ldots,x_{\lambda'_1},y_1,\ldots,y_{\lambda_1}$.
Select the starting square $(i,j) \in [\lambda]$ with probability
proportional to $x_i y_j$. At each step, move from square $(i,j)$ to a
random square in the modified hook so that the probability of moving to the square $(k,j)$ is proportional
to $x_{k}$, and the probability of moving to the square $(i,l)$ is proportional to $y_{l}$. When we reach a corner,
the process ends. We call this a \emph{modified weighted hook walk}.

If $\lambda$ has $c$ corners, there are $c$ different parts of $\lambda$, and also $c$
different parts of $\lambda'$. Take the ordered set partition $(U_1,\ldots,U_c)$ of the
set $\{1,2,\ldots,\lambda'_1\}$ so that $i$ and $j$ are in the same subset if and only
$\lambda_i = \lambda_j$, and so that the elements of the set $U_k$ are smaller than the
elements of the set $U_l$ if $k < l$. Then define $X_k$ as the sum of the elements of $U_k$. Similarly,
take the ordered set partition $(V_1,\ldots,V_c)$ of the set $\{1,2,\ldots,\lambda_1\}$
so that $i$ and $j$ are in the same subset if and only $\lambda'_i = \lambda'_j$, and
so that the elements of the set $V_k$ are smaller than the elements of the set $V_l$ if $k < l$.
Then define $Y_k$ as the sum of the elements of $U_k$.

In the example given in Figure \ref{f:XY}, we have $X_1 = x_{1}+x_{2}+x_{3}+x_{4}, X_2 = x_{5}+x_{6}+x_{7}, X_3 = x_{8}+x_{9}+x_{10}, X_4 = x_{11}+x_{12}+x_{13}+x_{14}+x_{15}+, X_5 = x_{16}$, $Y_1 = y_{1}+y_{2}, Y_2 = y_{3}+y_{4}+y_{5}+y_{6}, Y_3 = y_{7}+y_{8}+y_{9}+y_{10}+y_{11}, Y_4 = y_{12}+y_{13}+y_{14}+y_{15}+y_{16}+y_{17}+y_{18}$, and $Y_5 = y_{19}+y_{20}$.

\begin{figure}[hbt]

\begin{center}
\epsfig{file=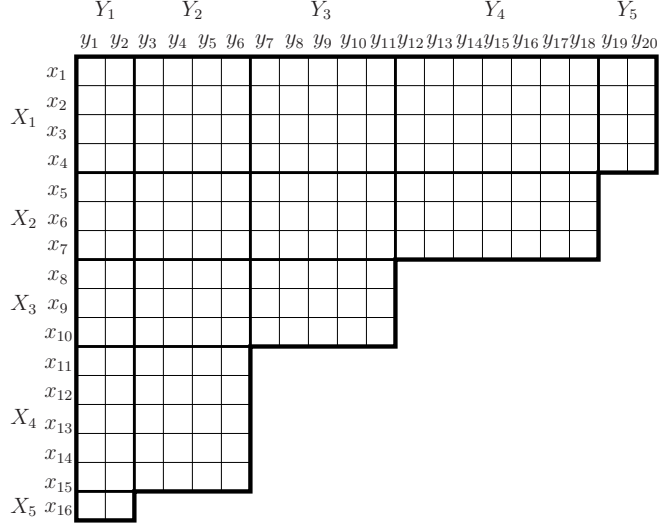,height=7cm}
\end{center}
\caption{The partition $(20,20,20,20,18,18,18,11,11,11,6,6,6,6,6,2)$ and corresponding sums $X_1,\ldots,X_5,Y_1,\ldots,Y_5$.}
\label{f:XY}
\end{figure}

Let us number the corners so that the top right corner is the first and the bottom left corner is the last.

\begin{thm}
 The probability that a modified weighted hook walk ends in corner $r$ is equal to
 $$\frac {X_r \ts Y_s}{\sum_{(p,q) \in [\la]} \. x_p \ts y_q}  
 \prod_{i=1}^{r-1} \. \left({\textstyle 1+ \frac{X_i}{X_{i+1}+\ldots + X_{r} \ts + \ts
 Y_{s+1} + \ldots + Y_{c+1-i}}}\right) 
 \prod_{j=1}^{s-1} \. \left({\textstyle 1+ \frac{Y_j}{X_{r+1}+\ldots + X_{c+1-j} \ts + \ts
 Y_{j+1} + \ldots + Y_{s}}}\right),$$
 where $s = c + 1 - r$.
\end{thm}
\begin{proof}
 Observe that the modified weighted hook walk is equivalent to the (ordinary) weighted hook
 walk on the staircase shape $(c,c-1,\ldots,1)$, where the $k$-th row is weighted by the sum$X_k$,
 and the $l$-th column is weighted by the sum $Y_l$. The formula then follows  from Theorem
 \ref{thm:whw} and the equality $\sum_{p+q \leq c+1} X_p \ts Y_q =
 \sum_{(p,q) \in [\la]} \. x_p \ts y_q$.
\end{proof}

\vskip.8cm

\section{The hook walk proof}\label{s:walk-proof}

What follows is an adaptation of the Greene-Nijenhuis-Wilf proof \cite{GNW}. Assume that the random process is $(i_1,j_1) \to (i_2,j_2) \to \ldots \to (r,s)$.
Then let $I = \{i_1,i_2,\ldots,r\}$ and $J = \{j_1,j_2,\ldots,s\}$ be its vertical and horizontal projections.

\begin{lemma} \label{lemma}
 The probability that the vertical and horizontal projections are
 $I$ and $J$, conditional on starting at $(i_1,j_1)$, is
 $$ \frac {\prod_{i \in I \setminus \{i_1 \}}x_i} {\prod_{i \in I \setminus \{r \}} (x_{i+1}+\ldots + x_{r} \ts + \ts y_{s+1} + \ldots + y_{\lambda_i})} \cdot  \frac {\prod_{j \in J \setminus \{j_1\}}y_j} {\prod_{j \in J \setminus \{s\}} (x_{r+1}+\ldots + x_{\lambda'_j} \ts + \ts
 y_{j+1} + \ldots + y_{s})}.$$
\end{lemma}

The lemma implies Theorem \ref{thm:whw}. Indeed, if we denote by $S$ the starting corner and by $F$ the final corner of the hook walk, then
$$\PP\big(F = (r,s)\big) = \sum_{(i_1,j_1) \in [\la]} \PP\big(S = (i_1,j_1)\big) \cdot \PP\big(F = (r,s) | S =  (i_1,j_1)\big) =$$
$$\sum_{i_1,j_1} \frac{x_{i_1}y_{j_1}}{\sum_{(p,q) \in [\la]} \. x_p \ts y_q}\left[{\textstyle \sum \frac {\prod_{i \in I \setminus \{i_1 \}}x_i} {\prod_{i \in I \setminus \{r \}} (x_{i+1}+\ldots + x_{r} \ts + \ts y_{s+1} + \ldots + y_{\lambda_i})} \cdot  \frac {\prod_{j \in J \setminus \{j_1\}}y_j} {\prod_{j \in J \setminus \{s\}} (x_{r+1}+\ldots + x_{\lambda'_j} \ts + \ts
  y_{j+1} + \ldots + y_{s})}}\right],$$
where the last sum is over $I,J$ satisfying $i_1 = \min I$, $r = \max I$, $j_1 = \min J$, $s = \max J$. Since
$$x_{i_1} \cdot \prod_{i \in I \setminus \{i_1 \}}x_i = x_r \cdot \prod_{i \in I \setminus \{r \}}x_i \qquad \mbox{and} \qquad y_{j_1} \cdot \prod_{j \in J \setminus \{j_1 \}}y_j = y_s \cdot \prod_{j \in J \setminus \{s \}}y_j,$$
this is equal to
$$\frac {x_r \ts y_s}{\sum_{(p,q) \in [\la]} \. x_p \ts y_q} \cdot \left[\sum \prod_{i \in I \setminus \{r\}} \. {\textstyle \frac{x_i}{x_{i+1}+\ldots + x_{r} \ts + \ts
y_{s+1} + \ldots + y_{\lambda_i}}} \cdot \prod_{j \in J \setminus \{s\}} \. {\textstyle \frac{y_j}{x_{r+1}+\ldots + x_{\lambda'_j} \ts + \ts
y_{j+1} + \ldots + y_{s}}}\right],$$
where the sum is over all $I,J$ with $r = \max I$, $s = \max J$. It is clear that this last product equals
$$\prod_{i=1}^{r-1} \. \left({\textstyle 1+ \frac{x_i}{x_{i+1}+\ldots + x_{r} \ts + \ts
y_{s+1} + \ldots + y_{\lambda_i}}}\right) \times \prod_{j=1}^{s-1} \. \left({\textstyle 1+ \frac{y_j}{x_{r+1}+\ldots + x_{\lambda'_j} \ts + \ts
y_{j+1} + \ldots + y_{s}}}\right).$$

\begin{proof}[Proof of Lemma \ref{lemma}]
 The proof is by induction on $|I| + |J|$. Denote the claimed probability by $\prod$. If $I = \{r\}$ and $J = \{s\}$, the probability is indeed $1$.
 For $|I| + |J| > 2$, we have
 $$\PP\big(I,J | S =  (i_1,j_1)\big) $$
 $$ = \frac{x_{i_2}}{x_{i_1+1}+\ldots + x_{\lambda'_{j_1}} \ts + \ts y_{j_1+1} + \ldots + y_{\lambda_{i_1}}} \cdot \PP\big(I \setminus \{i_1\},J | S = (i_2,j_1)\big)$$
 $$+ \frac{y_{j_2}}{x_{i_1+1}+\ldots + x_{\lambda'_{j_1}} \ts + \ts y_{j_1+1} + \ldots + y_{\lambda_{i_1}}} \cdot \PP\big(I ,J \setminus \{j_1\}| S =  (i_1,j_2)\big).$$
 By the induction hypothesis,
 $$\PP\big(I \setminus \{i_1\},J | S = (i_2,j_1)\big) = \frac{x_{i_1+1} + \ldots + x_r \ts + \ts y_{s+1} + \ldots + y_{\lambda_{i_1}}}{x_{i_2}} \prod,$$
 $$\PP\big(I,J \setminus \{j_1\} | S = (i_1,j_2)\big) = \frac{x_{r+1} + \ldots + x_{\lambda'_{j_1}} \ts + \ts y_{j_1+1} + \ldots + y_s}{y_{j_2}} \prod.$$
 Because $(x_{i_1+1} + \ldots + x_r \ts + \ts y_{s+1} + \ldots + y_{\lambda_{i_1}}) + (x_{r+1} + \ldots + x_{\lambda'_{j_1}} \ts + \ts y_{j_1+1} + \ldots + y_s) = x_{i_1+1}+\ldots + x_{\lambda'_{j_1}} \ts + \ts y_{j_1+1} + \ldots + y_{\lambda_{i_1}}$,
 it follows that $\PP\big(I,J | S =  (i_1,j_1)\big) = \prod$, which completes the proof.
\end{proof}

\section{Final remarks}\label{s:rem}

\subsection{ }\, \label{ss:rem-mot}
As Knuth wrote in~1973,  ``Since the hook-lengths formula is
such a simple result, it deserves a simple proof ...''
(see p.~63 of the first edition of~\cite{Knu}, cited also
in~\cite{Zei}).  Unfortunately, the desired simple proofs
have been sorely lacking.  It is our hope that
Section~\ref{s:bij} can be viewed as one such proof.

\subsection{ } \, \label{ss:rem-hist}
Surveying the history of the hook length formula is a
difficult task, even if one is restricted to purely
combinatorial proofs.  This is further complicated by the
ambiguity of the notions, since it is often unclear whether
a given technique is bijective or even combinatorial.  Below
we give a brief outline of some important developments,
possibly omitting a number of interesting and related
papers.\footnote{We apologize in advance to the authors of the
papers we do not mention; the literature is simply too big to
be fully surveyed here.}

The first breakthrough in the understanding of the role
of hooks was made by Hillman and Grassl in~\cite{HG}, where
they proved the (special case of) Stanley hook-content
formula by an elegant bijection.  It is well known that
this formula implies the hook length formula via the
$P$-partition theory~\cite[$\S 4$]{Sta} (see also~\cite{Pak}).
This approach was further developed
in \cite{BD,Gan2,KP,Kra1,Kra2,Kra3}.  Let us mention also
papers~\cite{Gan1,Pak}, where the connection to the
Robinson-Schensted-Knuth correspondence (see e.g.~\cite[$\S 7$]{Sta})
was established, and a recent follow up~\cite{BFP} with further
variations and algorithmic applications.

The next direction came in~\cite{GNW}, where an inductive proof
was established based on an elegant probabilistic argument.
This in turn inspired a number of further developments, including
\cite{GH,GNW-another,Ker1,Ker2}, and most recently~\cite{CLPS,CKP}.
In fact, the underlying hook length identities leading to
the proof have been also studied directly, without the probabilistic
interpretation; we refer to~\cite{Ver} and later developments~\cite{Ban,GN,Ker2,Kir}.
Needless to say, our two proofs can be viewed as direct descendants
of these two interrelated approaches.

As we mentioned in the
introduction, an important breakthrough was made by Zeilberger, who
found a ``direct bijectation'' of the GNW hook walk proof~\cite{Zei}.
In fact, his proof has several similar bijective steps as our
proof, but differs in both in technical details and the general scheme, being an
involved bijection of~(HLF) rather than~(BRHL).

Historically, the first bijective proof of the hook length formula
is due to Remmel~\cite{Rem} (see also~\cite{RW}). Essentially, he
uses the standard algebraic proof of Young (of the Frobenius-Young
product formula for $\dim \pi_\la$) and the Frame-Robinson-Thrall
argument, and replaces each step with a bijective version (sometimes
by employing new bijections and at one key step he uses the
Gessel-Viennot involution on intersecting paths~\cite{GV}).  He then
repeatedly applies the celebrated Garsia-Milne involution
principle to obtain an ingenious but completely intractable bijection
(a related approach was later outlined in~\cite{GV} as well).

Finally, there are two direct bijective proofs of the hook length
formula:~\cite{FZ} and~\cite{NPS}, both of which are highly non-trivial,
with the second using a variation on the jeu-de-taquin algorithm
(see~\cite[$\S 7$]{Sta}).  We refer to~\cite{Sag} for a nice
and careful presentation of the NPS bijection, and
to~\cite{Knu} for an elegant concise version.

\subsection{ } \, There are several directions in which our results
can be potentially extended.  First, it would be interesting to obtain the
analogues of our results for the  \emph{shifted} Young diagrams
and Young tableaux, for which there is an analogue of the
hook length formula due to Thrall~\cite{Thr} (see also~\cite{Sag}).
We refer to~\cite{Ban,Fis,Kra1,Sri} for other proofs
of the HLF in this case, and, notably, to~\cite{Sag-shifted}
for the shifted hook walk proof. We intend to return to this problem in the future.
Let us mention that a weighted version of the branching rule for trees is completely
straightforward.

Extending to semi-standard and skew tableaux is another possibility,
in which case one would be looking for a weighted analogue of
Stanley's hook-content formula~\cite{Sta} (see also~\cite{Mac}).

In a different direction, the weighted analogue of the ``complementary hook walk''
in~\cite{GNW-another} was discovered recently by the second author~\cite{kon}. 
The paper~\cite{GNW-another} is based on the observation that the Burnside identity
$$\sum_{\lambda \vdash n} |\syt(\lambda)|^2 = n!$$
is equivalent to the identity
$$\prod_{\bs \in [\la]} \, (h_\bs+1) \, = \,
\sum_{(r,s) \in \cor'[\la]} \.  \prod_{i=1}^{r-1} \. h_{is}
\. \prod_{j=1}^{s-1} \. h_{rj} \, \prod_{\bs \in \dor'_{rs}[\la]}
(h_\bs+1) \,,$$
where $\cor'[\la]$ is the set of squares $(r,s)$ that can be
added to the diagram of $\lambda$ so that the result is still a
diagram of a partition (in other words, $\cor'[\la]$ are the corners
of the complementary partition), and
$$\dor'_{rs}[\la] \, = \, \{(i,j)\in [\la], \
\text{such that} \ i \ne r, \, j \ne s\}\ts .
$$
In~\cite{kon}, the following complementary weighted branching rule is proved:
$$
\aligned
&  \prod_{(i,j) \in [\la]} \, \left(x_{i} + \ldots + x_{\la_j'}
+ \. y_{j}+ \ldots + y_{\la_i}\right) \\
&  \, = \, \sum_{(r,s) \in \cor'[\la]} \left[
\prod_{(i,j) \in \dor'_{rs}[\la]} \, \left(x_{i} + \ldots + x_{\la_j'}
+ \. y_{j}+ \ldots + y_{\la_i}\right)\right]\\
& \, \times \,  \left[ \prod_{i=1}^{r-1} \, \left(x_{i+1} + \ldots + x_{r}
+ \. y_{s}+ \ldots + y_{\la_i}\right)\right]  \cdot
\left[ \prod_{j=1}^{s-1} \, \left(x_r + \ldots + x_{\la_j'} + y_{j+1} + \ldots + y_{s}
 \right)\right]
\endaligned
$$

Let us note that although the $(q,t)$-hook walk defined in~\cite{GH} 
has several similarities, in full generality it is not a special case of 
the weighted hook walk.  While this might seem puzzling, let us emphasize
that the walks come from algebraic constructions of a completely different 
nature.  In many ways, it is much more puzzling that the
algebraic part of~\cite{CKP} is related to the branching rule at all.  

Finally, let us mention several new extensions of the hook
length formula recently introduced by Guo-Niu Han in~\cite{Han1,Han2}.  
There is also a hook walk style proof of the main identity in~\cite{CLPS}, 
which suggests a possibility of a ``weighted'' generalization.

\subsection{ } \, As we mentioned in the introduction, this paper
extends the results in our previous paper~\cite{CKP}, where we gave a
combinatorial proof of the following delicate result in the
enumerative algebraic geometry. \.
Denote by \. $w_\bs \ts = \ts i \ts \al + j \ts \beta$ \.
the weight of
a square \. $\bs=(i,j) \in \la$ \. in a Young diagram~$\la$.  Then:
$$
\sum_{\bs \in [\lambda]} \, w_\bs \. \cdot \.
\prod_{\bt \in [\lambda^\bs]}  \,
\frac{(w_\bt-w_\bs-\alpha)(w_\bt-w_\bs-\beta)}{(w_\bt-w_\bs-\alpha-\beta)(w_\bt-w_\bs)}
\, \. = \, n \. (\alpha + \beta)\.,
$$
where the product is over all squares in $[\la^\bs]$, defined as the Young diagram~$[\la]$
without squares \ts $\bs=(i,j)$  and~$(i+1,j+1)$, at which the denominator vanishes.
We refer to~\cite{CKP} for an explicit substitution which allows us to derive this
formula from~(WHL).

In a similar direction, we can obtain formulas corresponding to identities (b)--(d)
in Theorem~\ref{t:main}. We present them here without a proof.  Denote by $m = \la_1$ and
$\ell = \la_1'$ the lengths of the first row and the first column of~$[\la]$, respectively.
Then $w_{m\ts 0} = \la_1\ts\al$, $w_{0\ts \ell} = \la_1'\ts \be$, and
we have:
$$
\sum_{\bs \in [\lambda]} \, \frac{w_\bs}{w_\bs-w_{m\ts 0}}
\. \cdot \.
\prod_{t \in [\lambda^\bs]}  \,
\frac{(w_\bt-w_\bs-\alpha)(w_\bt-w_\bs-\beta)}{(w_\bt-w_\bs-\alpha-\beta)(w_\bt-w_\bs)}
\, \. = \,m \ts \left(1 + \frac{\alpha}\beta\right),
$$
$$
\sum_{\bs \in [\lambda]} \, \frac{w_\bs}{w_\bs-w_{0\ts \ell}}
\. \cdot \.
\prod_{\bt \in [\lambda^\bs]}  \,
\frac{(w_\bt-w_\bs-\alpha)(w_\bt-w_\bs-\beta)}{(w_\bt-w_\bs-\alpha-\beta)(w_\bt-w_\bs)} \, \.
= \, \ell \ts \left(1 + \frac{\be}\al\right),
$$
$$
\sum_{\bs \in \lambda} \frac{w_\bs}{(w_\bs-w_{m\ts 0})(w_\bs-w_{0 \ts \ell})}
\. \cdot \.
\prod_{\bt \in \lambda^\bs}  \,
\frac{(w_\bt-w_\bs-\alpha)(w_\bt-w_\bs-\beta)}{(w_\bt-w_\bs-\alpha-\beta)(w_\bt-w_\bs)} \, = \,
\frac{1}\alpha \. + \. \frac{1}\beta \,\..
$$
It would be interesting to understand the role of these formulas in the algebraic context.

\vskip.6cm

\noindent
{\bf Acknowledgements} \ The authors are grateful to Dennis Stanton for pointing
out~\cite{CLPS} to us and explaining its inner working. I.C-F. would like to 
thank the Korean Institute for Advanced Studies for support and
excellent working conditions. M.K. would like to thank Paul Edelman
for several helpful comments on an early draft of this paper. I.P. would like to
thank Persi Diaconis for teaching him Kerov's ``segment splitting'' algorithm.
Partial support for I.C.-F. from the NSF under the grant DMS-0702871 is gratefully
acknowledged. I.P. was partially supported by the NSF and the NSA.

\end{document}